\documentclass[11pt,reqno]{amsart}

\usepackage[a4paper,left=35mm,right=35mm,top=30mm,bottom=30mm,marginpar=25mm]{geometry}

\usepackage{amsmath}
\usepackage{amssymb}
\usepackage{amsthm}
\usepackage{eurosym}
\usepackage{graphicx}
\usepackage{epsfig}
\usepackage{hyperref}
\usepackage{dsfont}
\usepackage[displaymath,mathlines]{lineno}
\usepackage{subfigure}
\allowdisplaybreaks

\usepackage{hyperref}

\usepackage{ifthen}
\usepackage{ esint }
\makeindex

\newcommand{\Rr}{{\mathbb{R}}}

\newcommand{\Tt}{{\mathbb{T}}}

\def\leq{\leqslant}
\def\geq{\geqslant}

\numberwithin{equation}{section}

\newtheoremstyle{thmlemcorr}{10pt}{10pt}{\itshape}{}{\bfseries}{.}{10pt}{{\thmname{#1}\thmnumber{
#2}\thmnote{ (#3)}}}
\newtheoremstyle{thmlemcorr*}{10pt}{10pt}{\itshape}{}{\bfseries}{.}\newline{{\thmname{#1}\thmnumber{
#2}\thmnote{ (#3)}}}
\newtheoremstyle{defi}{10pt}{10pt}{\itshape}{}{\bfseries}{.}{10pt}{{\thmname{#1}\thmnumber{
#2}\thmnote{ (#3)}}}
\newtheoremstyle{remexample}{10pt}{10pt}{}{}{\bfseries}{.}{10pt}{{\thmname{#1}\thmnumber{
#2}\thmnote{ (#3)}}}
\newtheoremstyle{ass}{10pt}{10pt}{}{}{\bfseries}{.}{10pt}{{\thmname{#1}\thmnumber{
A#2}\thmnote{ (#3)}}}

\theoremstyle{thmlemcorr}
\newtheorem{theorem}{Theorem}
\numberwithin{theorem}{section}
\newtheorem{lemma}[theorem]{Lemma}

\newtheorem{proposition}[theorem]{Proposition}

\theoremstyle{thmlemcorr*}
\newtheorem{theorem*}{Theorem}
\newtheorem{lemma*}[theorem]{Lemma}
\newtheorem{corollary*}[theorem]{Corollary}
\newtheorem{proposition*}[theorem]{Proposition}
\newtheorem{problem*}[theorem]{Problem}
\newtheorem{conjecture*}[theorem]{Conjecture}

\theoremstyle{defi}

\theoremstyle{remexample}
\newtheorem{remark}[theorem]{Remark}

\newtheorem{teo}[theorem]{Theorem}
\newtheorem{lem}[theorem]{Lemma}

\newtheorem{cor}[theorem]{Corolary}

\theoremstyle{ass}

\usepackage{color}

\begin{document}

\title{One-dimensional, forward-forward  mean-field games with congestion}
 
\author{Diogo A.
        Gomes}         \address[D. A. Gomes]{
        King Abdullah University of Science and Technology (KAUST), CEMSE Division , Thuwal 23955-6900. Saudi Arabia.}

\author{Marc Sedjro}

\address[Marc Sedjro]{
        King Abdullah University of Science and Technology (KAUST), CEMSE
        Division , Thuwal 23955-6900. Saudi Arabia.}

\subjclass[2010]{} 

\thanks{
        The authors were supported by King Abdullah University of Science and Technology (KAUST) baseline and start-up funds. 
}
\date{\today}

\begin{abstract}
Here, we consider one-dimensional forward-forward mean-field games (MFGs)
with congestion, which 
were introduced to approximate stationary
MFGs. 
We use methods from the theory 
of conservation laws to examine the qualitative properties
of these games. First, by computing Riemann invariants and corresponding invariant regions,  
we develop a method to prove lower bounds for the density.
Next, by combining the lower bound with an entropy function, we prove 
the existence of global solutions for parabolic forward-forward MFGs. 
Finally, we construct traveling-wave solutions, which 
settles in a negative way the convergence problem
for 
forward-forward MFGs.
A similar technique gives the existence of time-periodic solutions for non-monotonic MFGs. 
\end{abstract}

\maketitle

\section{Introduction}
 Mean-field games (MFGs)  model competitive interactions in large  populations in which the agents' actions depend on statistical information about the population. These games are modeled by a Hamilton-Jacobi equation coupled with a Fokker-Planck equation. An important class of MFGs concerns congestion problems, where
 the agents' motion  is hampered in high-density areas.
    
 In one-dimension, the congestion problem is the system
 \begin{equation}\label{eq: cbf mfg}
\begin{cases}
-u_t+ m^{\alpha}H\left( \frac{u_x}{m^{\alpha}}\right) =\varepsilon u_{xx}  + g(m)\\
m_t-\left(H'\left( \frac{u_x}{m^{\alpha}}\right) m\right)_x= \varepsilon m_{xx}, 
\end{cases}
\end{equation}
together
 with  terminal-initial conditions;  we prescribe the initial value of $m$ at $t=0$ and the terminal value of $u$, at $t=T$:
 \begin{equation}
 \label{itc}
 \begin{cases}
 u(x,T) = u_T(x) \\
 m(x,0) = m_0(x).
 \end{cases}
 \end{equation}
Here, the Hamiltonian, $H:\Rr\to \Rr$, $H\in C^2$, is convex and bounded by below, the coupling, $g:\Rr^+\to \Rr$, $g\in C^1$, is a monotone increasing function, the viscosity coefficient, $ \varepsilon$, is a non-negative parameter and $\alpha$, $0<\alpha<2$ is the congestion exponent.   
The space variable, $x$, lies on the one-dimensional torus, $\Tt$, identified with $[0,1]$; the time variable, $t$, belongs to $[0,T]$ for some terminal time, $T>0$. The unknowns in the above system are  $u:\Tt\times [0,T]\to \Rr$ and $m:\Tt\times [0,T]\to \Rr^+$, 
where for each $t>0$, $m$ is a probability density; that is, 
\[
\int_{\Tt} m(x, t)dx =1. 
\]
In MFGs, each agent follows a path determined by a control problem.  
For $\varepsilon=0$, these paths are deterministic, whereas for $ \varepsilon>0$ they are subject to random noise. If  $\varepsilon=0$ then \eqref{eq: cbf mfg} is a  {\em first-order MFG} otherwise the system is a {\em parabolic MFG}.
For an agent  at  $x\in \Rr$ at time $t$, the quantity $u(x,t)$ is the value function of a control problem in which  
 the  Hamiltonian, $H\in C^2$, determines the cost and $g\in C^1$ provides a coupling between each agent and the mean field, $m$. The evolution of $m$ through the Fokker-Planck equation, the second equation in (\ref{eq: cbf mfg}), 
  also depends on $H$. 
  
MFGs, introduced in  \cite{ll1}, were investigated extensively in the last few years.  Numerous results on the existence and uniqueness of solution are now available. For example, the parabolic problem was considered for strong solutions in \cite{GPim2, GPim1, ll2} and weak solutions in \cite{ ll2, porretta2}. The first-order problem was tackled in  \cite{Cd2, GraCard} and weak solutions are obtained. As for the stationary problem, classical and weak solution were sought in \cite{GMit,GP,GPat,GPatVrt}. In both standard and stationary problems, the uniqueness of solution is guaranteed by monotonicity of the coupling $g$. The congestion problem was first addressed in \cite{LCDF}  and later, other approaches such as density constraints and nonlinear mobilities were used in \cite{MR3199781, FrS, San12}.  The  existence of smooth solutions for a small terminal time was discussed in \cite{GVrt2} and in \cite{Graber2}.   
  
 {\em Forward-forward MFGs}  result from reversing the time in the Hamilton-Jacobi equation in \eqref{eq: cbf mfg}.
 Here, we focus on the forward-forward  MFGs with congestion:
 \begin{equation}\label{eq: foba mfg system2}
 \begin{cases}
 u_t+ m^{\alpha}H\left( \frac{u_x}{m^{\alpha}}\right) =\varepsilon u_{xx}  + g(m)\\
 m_t-\left(H'\left( \frac{u_x}{m^{\alpha}}\right) m\right)_x= \varepsilon m_{xx}
 \end{cases}
 \end{equation}
with the  {\em initial conditions}:  
 \begin{equation}\label{ini-ini}
 \begin{cases}
 u(x,0) = u_0(x) \\
 m(x,0) = m_0 (x).
 \end{cases}
 \end{equation}

In \cite{DY}, 
the authors propose
the forward-forward MFG model
and study numerically its convergence to 
a stationary MFG. 
This scheme relies on the parabolicity in (\ref{eq: cbf mfg}) to force 
the long-time convergence to a stationary solution.
In \cite{GPff}, the authors studied
parabolic forward-forward MFGs and proved the existence of a solution. Next, in  \cite{GomNurSed} using the entropy method, the convergence for one-dimensional forward-forward MFGs without congestion was proven. Additionally, there, the authors compute entropies for first-order problems and establish connections between these models and certain nonlinear wave equations from elastodynamics. 

In \cite{GPff}, the authors proved the existence and regularity for parabolic forward-forward MFGs with  subquadratic Hamiltonians. In \cite{llg2}, the authors investigated forward-forward MFGs  with logarithmic coupling in the framework of eductive stability.  The large-time convergence was studied in \cite{GomNurSed} for one-dimensional parabolic problems and numerical evidence from  \cite{AcCiMa} and \cite{Ci} suggests that convergence holds.

The main contributions of this paper are as follows. First, for quadratic Hamiltonians, we convert the forward-forward problem into a system of conservation laws and compute new convex entropies (Lemma \ref{elem}) and Riemann invariants (Lemmas \ref{lem:riemann invariant z} and \ref{lem:riemann invariant w}).
These Riemann invariants give lower bounds for the density, $m$, and, 
for parabolic MFGs, these bounds combined with an entropy estimate gives the existence of a classical global solution (Theorem \ref{mteo}). Finally, by computing 
traveling-wave solutions, we prove that forward-forward MFGs may fail to 
converge to a stationary solution. With a similar method, we construct traveling waves for non-monotonic MFGs.

\section{Systems of conservation laws and first-order, forward-forward MFGs}

Here, we study the following forward-forward MFG with congestion and a quadratic Hamiltonian: 
\begin{equation}\label{eq: cff mfg }
\begin{cases}
u_t+\frac{(p+u_x)^2}{2m^{\alpha}}=0,\\
m_t-\left( \frac{p+u_x}{m^{\alpha-1}}\right)_x=0,
\end{cases}
\end{equation}
with $p\in \Rr$. 
 Assuming enough regularity, we differentiate the first equation with respect to $x$ and set $v=p+u_x$. We thus obtain the following system of conservation laws:
\begin{equation}
\label{eq: cff mfg_ cl}
\begin{cases}
v_t+\left( \frac{v^2}{2m^{\alpha}}\right) _x=0,\\
m_t-\left( \frac{v}{m^{\alpha-1}}\right)_x=0.
\end{cases}
\end{equation}
\subsection{Existence of convex entropies}
First, we construct convex entropies for \eqref{eq: cff mfg_ cl}. We recall that $(\eta, q)$ is an entropy/entropy-flux pair for  (\ref{eq: cff mfg_ cl}) if 
\begin{equation}\label{eq: entropy/entflux 0}
 \nabla\eta \nabla F=\nabla q,
\end{equation}
where
$$
F(v,m)= \left(\frac{v^2}{2m^{\alpha}},-\frac{v}{m^{\alpha-1}} \right)
$$
is the flux function in \eqref{eq: cff mfg_ cl}.

A direct computation shows that $\eta$ solves \eqref{eq: entropy/entflux 0}
if and only if it satisfies:

\begin{equation}
\label{eq: entropy/entflux 1}
\frac{\alpha v^2}{2m^{\alpha+1}}  \eta_{vv}-\frac{1}{m^{\alpha-1}} \eta_{mm} +\frac{(2-\alpha) v }{m^{\alpha}} \eta_{vm}=0.
\end{equation}

In the next lemma, we investigate the existence of entropies and determine conditions under which these  entropies are convex.  
\begin{lemma}
\label{elem}
        The system \eqref{eq: cff mfg_ cl} has the following family of entropies:
\begin{equation}
        \label{eq:entropy}
        \eta(v,m)= v^a m^b,
\end{equation} 
where $a, b$ satisfy 
\begin{equation}
\label{eq: cond on a,b for entropy}
\alpha a (-a+2 b+1)+2 b (-2 a+b-1)=0.
\end{equation}
Moreover,   if, additionally $a$ is even and
 \begin{equation}
 \label{eq: cond on a,b for convexity}
 a(a-1)\geq 0 \quad\hbox{and}\quad ab(a+b-1)\leq 0, 
 \end{equation}
then $\eta$ as defined in \eqref{eq:entropy} and \eqref{eq: cond on a,b for entropy}  is convex.
\end{lemma}
\proof
By inspection, we see that  $\eta$ as defined in \eqref{eq:entropy} satisfies (\ref{eq: entropy/entflux 1}) if and only if \eqref{eq: cond on a,b for entropy} holds.

Note that 
\[
\ D^2\eta=\left( 
\begin{array}{cc}
(a-1) a m^b v^{a-2} & a b m^{b-1} v^{a-1} \\
a b m^{b-1} v^{a-1} & (b-1) b m^{b-2} v^a \\
\end{array}\right).
\]
Using Sylvester's criterion, that $a$ is even, and that $m>0$, 
we derive \eqref{eq: cond on a,b for convexity}.
\endproof

\endproof
\begin{remark}
\label{r22}	
For $0<\alpha<2$, the conditions \eqref{eq: cond on a,b for entropy} and \eqref{eq: cond on a,b for convexity} hold  if
$a<0$ or  $a>1$ and with $b$ given by

\begin{equation}
\label{b for convexity of eta}
b=b(\alpha, a):=\frac{ 1 + 2 a - a \alpha - \sqrt{
1 + 4 a - 4 a \alpha + a^2 (4 -2\alpha  + \alpha^2)}}{2}.
\end{equation}
We note that if $a>1$ then  $b<0$. Moreover, we have
\begin{equation}
\label{ limit of b/a}
\lim_{a\rightarrow\infty}\frac{-b(\alpha, a)}{a}=\theta(\alpha), 
\end{equation}
with
\begin{equation}
\label{tda}
\theta(\alpha)=
\frac{1}{2} \left(\sqrt{\alpha ^2-2 \alpha +4}+\alpha -2\right)
\end{equation}
and, for $a>1$,
\begin{equation}
\label{comparison of b/a and theta}
\theta(\alpha) -\frac{b(\alpha, a)}{a}>0.
\end{equation}

\end{remark}
\subsection{Hyperbolicity and Genuinely nonlinearity}
Now, we show that \eqref{eq: cff mfg_ cl}
is a hyperbolic, genuinely nonlinear system of conservation laws. 
 To that end,  we compute the  Jacobian of $F$ and get
\begin{equation}
     DF(v,m) =\begin{pmatrix} 
     m^{-\alpha}v & -\dfrac{1}{2}\alpha m^{-1-\alpha} v^2\\
    -m^{1-\alpha}& -(1-\alpha)m^{-\alpha} v
\end{pmatrix}.
\end{equation}

\begin{proposition} \label{prop: eigenvalues and eigenvectors}
The system \eqref{eq: cff mfg_ cl} is strictly hyperbolic outside the set $\left\lbrace v\neq 0 \right\rbrace $.
More precisely, \eqref{eq: cff mfg_ cl}  has eigenvalues 
 \begin{equation}
 \label{EigenVal1}
 \lambda_1=\left(\alpha -  \sqrt{4-2\alpha+\alpha^2}\right) \frac{v}{2m^{\alpha}}
 \end{equation}
 and 
\begin{equation}
\label{EigenVal2}
 \lambda_2=\left(\alpha + \sqrt{4-2\alpha+\alpha^2}  \;\right) \frac{v}{2m^{\alpha}},
 \end{equation}
with corresponding eigenvectors
\begin{equation}
\label{EigenVect1}
 \mathbf{r}_1=\left( (2 -\alpha )v -\sqrt{4-2\alpha+\alpha^2} \;v ,-2m\right)
 \end{equation}
 and 
\begin{equation}
\label{EigenVect2}
 \mathbf{r}_2=\left((2 -\alpha )v+\sqrt{4-2\alpha+\alpha^2}\; v ,-2m\right). 
 \end{equation}
 Furthermore, for $0<\alpha<2$, the system \eqref{eq: cff mfg_ cl}
 is  genuinely nonlinear on
 \begin{equation}
 \label{adef}
 \mathcal{A}=\left\lbrace (v,m)\in\mathbb{R}^2: v>0,\; m>0\right\rbrace. 
 \end{equation}  
\end{proposition}      
\begin{remark}
The set $ \mathcal{A}$ corresponds to the case where $p x+u(x)$ is increasing.  
Analogous results hold for decreasing functions. 
\end{remark}
\proof
Simple computations  show that $DF$ has eigenvalues given by \eqref{EigenVal1} and \eqref{EigenVal2}.  These are distinct if $v\neq 0$.

Thus, \eqref{eq: cff mfg }  is a strictly hyperbolic system of conservation laws
if $v\neq 0$. 
Next, we find the right eigenvectors corresponding to $\lambda_1$ and $\lambda_2$. Accordingly, 
 we determine $\mathbf{r}_i$, $i=1,2$, such that 
 \[
         DG^T \mathbf{r}_i=\lambda_i \mathbf{r}_i.
 \]
 Again, straightforward computations ensure that  $\mathbf{r}_1,$ $\mathbf{r}_2$ can be chosen as in \eqref{EigenVect1}  and \eqref{EigenVect2}.

Next, we compute
$$
        \nabla \lambda_1\cdot \mathbf{r}_1
        = \left(2+\alpha ^2  -\sqrt{(\alpha -2) \alpha +4}-\alpha \sqrt{(\alpha -2) \alpha +4}\right) v m^{-\alpha }
$$
and
\[
   \nabla \lambda_2\cdot \mathbf{r}_2
   =    \left(2+\alpha ^2+ \sqrt{(\alpha -2) \alpha +4}+\alpha  \sqrt{(\alpha -2) \alpha +4}\right) v m^{-\alpha }. 
\]

Finally, we observe that, for  $0<\alpha<2$,
we have 
\begin{equation}\label{eq : genuinly nonlinear 1}
\nabla \lambda_1\cdot \mathbf{r}_1\leq 0.
\end{equation}

 The equality in \eqref{eq : genuinly nonlinear 1} holds if and only if $v=0$. On the other hand, for any $\alpha\in \mathbb{R}$,
 \begin{equation}\label{eq : genuinly nonlinear 2}
 \nabla \lambda_2\cdot \mathbf{r}_2\geq 0.
 \end{equation}
Likewise, the equality in \eqref{eq : genuinly nonlinear 2} holds if and only if $v=0$. Thus, for 
 $0<\alpha<2$, the system is genuinely nonlinear on $\mathcal{A}$. 
 \endproof

\subsection{Riemann invariants}
Now, we compute Riemann invariants for the above system. Later, we show that solutions whose initial values take values in $\mathcal{A}$ remain in $\mathcal{A}$.
Set
 \begin{equation}\label{eq :  def A, B}
A(\alpha)=2 -\alpha  +\sqrt{4-2\alpha+\alpha^2}\quad\hbox{and}\quad B(\alpha)=2 -\alpha  -\sqrt{4-2\alpha+\alpha^2}
\end{equation}
 and  note that $A(\alpha)$ is a positive  for all $\alpha\in \mathbb{R}$ whereas $B(\alpha)$ is negative if and only $\alpha >0$.

  Moreover, 
 \begin{equation}
 B(\alpha)+2>0 \Longleftrightarrow\alpha<2.
 \end{equation}
 Define
 \[
\mathcal{S}_0:=\left\lbrace (s,\alpha): s<0,\; \alpha \geq 0 \right\rbrace
 \]
and 
  \[
   \mathcal{S}_1:=\left\lbrace (r,\alpha):   r<\frac{2 B(\alpha)}{B(\alpha)+2 } ,\; 0<\alpha<2\right\rbrace . 
  \]
 In the following Lemmas, we compute Riemann invariants
 for \eqref{eq: cff mfg_ cl}.
\begin{lemma}
        \label{lem:riemann invariant z}
        The family of functions 
        $$ z(v,m)= v^{\frac{s}{A(\alpha)}} m^{\frac{s}{2}},$$
        $s\in \mathbb{R}$ are Riemann invariants for \eqref{eq: cff mfg_ cl}. Moreover, if $(s,\alpha)\in \mathcal{S}_0$
        then $z$ is convex.
\end{lemma}
\proof
We recall that $z$ is a Riemann invariant corresponding to  $r_1$ if
 \[
 \nabla z\cdot r_1= 0.
 \]
 This means that 
 \begin{equation}
         \label{eq : diif z 1}
         A(\alpha)v\partial _{v}z- 2m\partial_{m}z =0
 \end{equation}
for
\[
        A(\alpha)=2 -\alpha  +\sqrt{4-2\alpha+\alpha^2}.
\]
Setting $z(v,m)=a(v) b(m)$,  \eqref{eq : diif z 1} becomes
\[
 A(\alpha)v a'(v) b(m) =2m a(v) b'(m), 
 \]
or 
\[ A(\alpha)v\dfrac{a'(v)}{a(v)}= 2 m\dfrac{b'(m)}{b(m)}.
\]
       
Thus, $a(v)= v^{s/A(\alpha)}$ and $b(m)= m^{s/2}$   for $s\in \mathbb{R}$. Therefore, 
\[
z(v,m)= v^{\frac{s}{A(\alpha)}} m^{\frac{s}{2}}.
\]
 To study the convexity of $z$, we compute its Hessian:
  
\[
D^2z(v,m)=
        \begin{pmatrix} 
                \frac{s \left(- A(\alpha)+s\right) m^{s/2} v^{\frac{s}{A(\alpha)}-2}}{A(\alpha)^2}&\frac{s^2 m^{\frac{s}{2}-1} v^{\frac{s}{A(\alpha)}-1}}{2A(\alpha)} \\
                 \frac{s^2 m^{\frac{s}{2}-1} v^{\frac{\lambda}{A(\alpha)}-1}}{2 A(\alpha)}&\frac{1}{4} (s-2) s m^{\frac{s}{2}-2} v^{\frac{s}{A(\alpha)}}
        \end{pmatrix}. 
\]
First, we observe that the trace of \(D^2z(v,m)\) has the sign of

\[
s \left[s \left(A^2(\alpha) v^2+4 m^2\right)-2 A^2(\alpha) v^2-4 A(\alpha) m^2\right].
\]
Thus, the trace is positive if only if
\begin{equation}
\label{trace_z positive}
       s <0 \qquad\hbox{or}\qquad s> \frac{2 A^2(\alpha) v^2+4 A(\alpha) m^2}{A^2(\alpha) v^2+4 m^2}.
\end{equation}
The determinant of $D^2z(v,m)$ has the sign of

\begin{equation}
\label{det_z}
-s^2\left[(A(\alpha)+2) s-2 A(\alpha)\right].
\end{equation}
Hence, the determinant is positive if and only if 
\begin{equation}
\label{det_z positive}
     s<\frac{2 A(\alpha)}{A(\alpha)+2 }.
\end{equation}
 Observe that 
 
 \begin{equation}
 \label{comparison for convexity of z}
 0<\frac{2 A(\alpha)}{A(\alpha)+2 }<\frac{2 A^2(\alpha) v^2+4 A(\alpha) m^2}{A^2(\alpha) v^2+4 m^2}
 \end{equation}
 for  all $\alpha\in\mathbb{R}$. In view of \eqref{trace_z positive}, \eqref{det_z positive} and \eqref{comparison for convexity of z}, if $(s,\alpha)\in \mathcal{S}_0$,  then $z$ is convex.
 \endproof
 \begin{lemma}
        \label{lem:riemann invariant w}
        Let $0<\alpha<2$. The family of functions 
        $$ w(v,m)= v^{\frac{r}{B(\alpha)}} m^{\frac{r}{2}},$$
        $r\in \mathbb{R}$ are Riemann invariants for \eqref{eq: cff mfg_ cl}. Moreover, if $(r,\alpha)\in \mathcal{S}_1$ 
        then $w$ is convex.
 \end{lemma}
 \proof If $w$ is a Riemann invariant  corresponding to  $\mathbf{r}_2$ then

\begin{equation}\label{eq : diif 2 1}
\nabla w\cdot \mathbf{r}_2= 0.
\end{equation}
Furthermore, we have \eqref{eq : diif z 1}
for $w(v,m)=a(v) b(m)$, with
 $a(v)= v^{r/B(\alpha) }$,  $b(m)= m^{r/2}$, 
and $r\in \mathbb{R}$. 

Moreover,
\[
D^2w(v,m)=\begin{pmatrix} 
 \frac{r \left(r- B(\alpha)\right) m^{r/2} v^{-\frac{r}{B(\alpha)}}}{\left(B(\alpha)\right)^2}&\frac{r^2 m^{\frac{r}{2}-1} v^{\frac{r}{B(\alpha)}-1}}{2 B(\alpha)}\\
   \frac{r^2 m^{\frac{r}{2}-1} v^{\frac{r}{B(\alpha)}-1}}{2 B(\alpha)}&\frac{1}{4} (r-2) r m^{\frac{r}{2}-2} v^{\frac{r}{B(\alpha)}}
\end{pmatrix}. 
\]
 Note that the first leading principal minor of $D^2w(v,m)$ has the same sign as 
 $$ r \left(r-B(\alpha)\right), $$
and is thus  positive if only if 
 \begin{equation}
 \label{trace_w positive}
 r < B(\alpha) \qquad\hbox{or}\qquad  r > 0.
 \end{equation}
 The determinant of $D^2w(v,m)$ has the sign of 
 
 \begin{equation}
 \label{det_w}
 -r^2\left[(B(\alpha)+2) r-2 B(\alpha)\right].
\end{equation}
    Thus, the determinant is positive if and only if 
    \begin{align*}
    r<\frac{2 B(\alpha)}{B(\alpha)+2 }.
    \end{align*}
    If $ (r,\alpha)\in \mathcal{S}_1$ then all the principal minors of the Hessian $D^2w(v,m)$ are positive. Accordingly, $w$ is convex.

\begin{figure}[ht]
        \centering
\includegraphics[width=5cm,height=5cm]{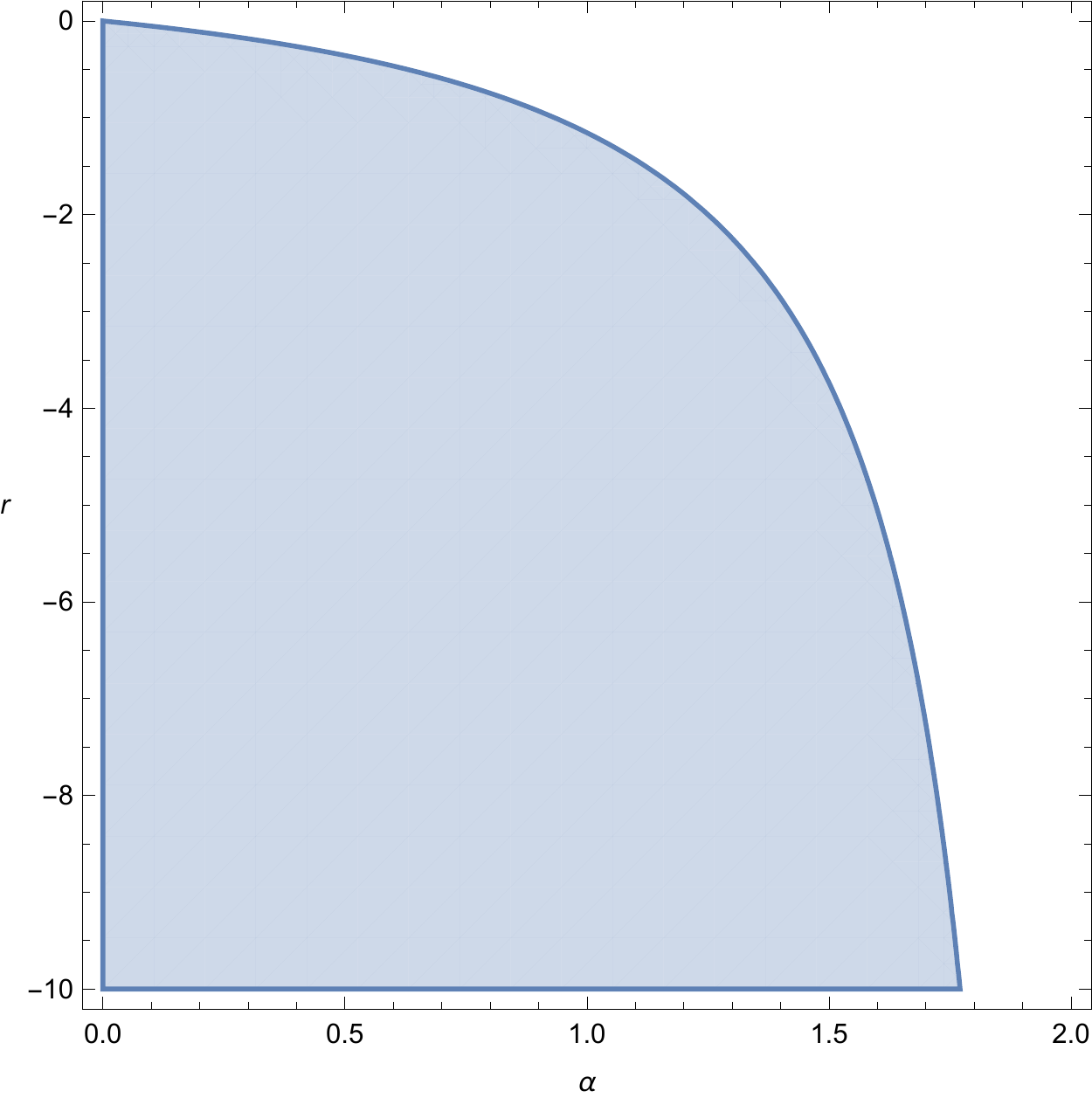}
        \caption{Domain $\mathcal{S}_1$.}
\end{figure}

\section{Parabolic forward-forward MFGs}

Now, we consider the parabolic forward-forward MFG corresponding to  (\ref{eq: cff mfg }): 
 \begin{equation}\label{eq: cff mfg vis}
 \begin{cases}
 v_t+\left( \frac{v^2}{2m^{\alpha}}\right) _x=\varepsilon v_{xx},\\
 m_t-\left( \frac{v}{m^{\alpha-1}}\right)_x=\varepsilon m_{xx},
 \end{cases}
 \end{equation} 
 with   $\varepsilon>0$ and initial conditions
  \begin{equation}
  \label{icc}
  \begin{cases}
  v(x,0) = v_0(x) \\
  m(x,0) = m_0(x),
  \end{cases}
  \end{equation}
  where $(v_0,m_0)$ takes values in  a compact subset, $\mathcal{K}$,  of $\mathcal{A}$, where, as before,  
 $$\mathcal{A}=\left\lbrace (v,m)\in\mathbb{R}^2: v>0,\; m>0\right\rbrace. $$
  
Standard PDE theory guarantees that  \eqref{eq: cff mfg vis}-\eqref{itc} has a unique classical solution $(v^{\varepsilon}, m^{\varepsilon})$ on $\Tt\times [0, T_{\infty})$ for some $0<T_{\infty}\leq \infty$.
Our goal is to prove that the maximal existence time is $T_{\infty}=\infty$. 

We recall the 
Riemann invariants from Lemmas \ref{lem:riemann invariant z} and  \ref{lem:riemann invariant w}
\begin{equation}
\label{fzm}
 z(v,m)= v^{\frac{s}{A(\alpha)}} m^{s}\quad \hbox{and}\quad  w(v,m)= v^{\frac{r}{B(\alpha)}} m^{r},
\end{equation}
     where $A, B$ are defined in \eqref{eq :  def A, B}. By the same Lemmas, $w$ and $z$ are convex if $(s,\alpha)\in \mathcal{S}_0$ and $(r,\alpha)\in \mathcal{S}_1$.
\begin{proposition}
\label{prop:invariant-dom}
        Assume that  $(s,\alpha)\in \mathcal{S}_0$ and $(r,\alpha)\in \mathcal{S}_1$ and that the initial conditions in \eqref{icc} for \eqref{eq: cff mfg vis}
        take values in a compact subset set, $\mathcal{K}$, of $\mathcal{A}$
        and
        satisfy $w(v_0, m_0) < M$ and $z(v_0, m_0) < M$ for some $M>0$. Then, the solution $(v^{\varepsilon}, m^{\varepsilon})$ of (\ref{eq: cff mfg vis}) satisfies 
 \[
  w(v^{\varepsilon}, m^{\varepsilon}) < M  \qquad\text{and} \qquad  z(v^{\varepsilon}, m^{\varepsilon}) < M
\]
for all $t\in [0,T_{\infty})$. 
Moreover,   
there exists $\tilde c_0$ such that 
   $m^{\varepsilon}(t,x)\geq \tilde c_0>0$ on $\mathbb{T}$ for $t\in [0,T_{\infty})$. Furthermore, $v(x,t)>0$
   for $t\in [0,T_{\infty})$.
\end{proposition}
\begin{proof}
First, using (\ref{eq: cff mfg vis}), we get
\begin{equation*}
        w(v,m)_t+\lambda_1 w(v,m)_x=\varepsilon\left(  w(v,m)_{xx}-(v_x,m_x)^T D^2w(v,m) (v_x,m_x)\right) 
\end{equation*}
and
\begin{equation*}
z_t(v,m)+\lambda_2 z_x(v,m)= \varepsilon\left(  z(v,m)_{xx}-(v_x,m_x)^T D^2z(v,m) (v_x,m_x)\right). 
\end{equation*}
 Here, $\lambda_1$ and $\lambda_2$ are eigenvalues as obtained in Proposition \ref{prop: eigenvalues and eigenvectors}. Since $z$ and $w$ are convex, we obtain 
\begin{align*}\label{eq: heat ineq w}
w(v,m)_t+\lambda_1 w(v,m)_x& \leq \varepsilon w(v,m)_{xx}\\
z(v,m)t+\lambda_2 z(v,m)_x& \leq \varepsilon z(v,m)_{xx}.
\end{align*}
 Finally, we use the maximum principle to get that, if $w(v_0, m_0) < M$ and $z(v_0, m_0) < M$ for some $M>0,$ then the solution, $(v^{\varepsilon}, m^{\varepsilon})$, of (\ref{eq: cff mfg vis}) satisfies
  $$w(v^{\varepsilon}, m^{\varepsilon}) < M  \qquad\text{and} \qquad  z(v^{\varepsilon}, m^{\varepsilon}) < M.$$ 
Next, we observe that $z,w<M$, with $z$ and $w$ given in \eqref{fzm},
implies that $m$ is bounded by below,
as can be seen from the level sets of $z$ and $m$
depicted in Figure \ref{lsts}.

        	\begin{figure}[ht]
        	\centering
        	\subfigure[Level sets of $z$.]{
        		\includegraphics[width=5cm,height=5cm]{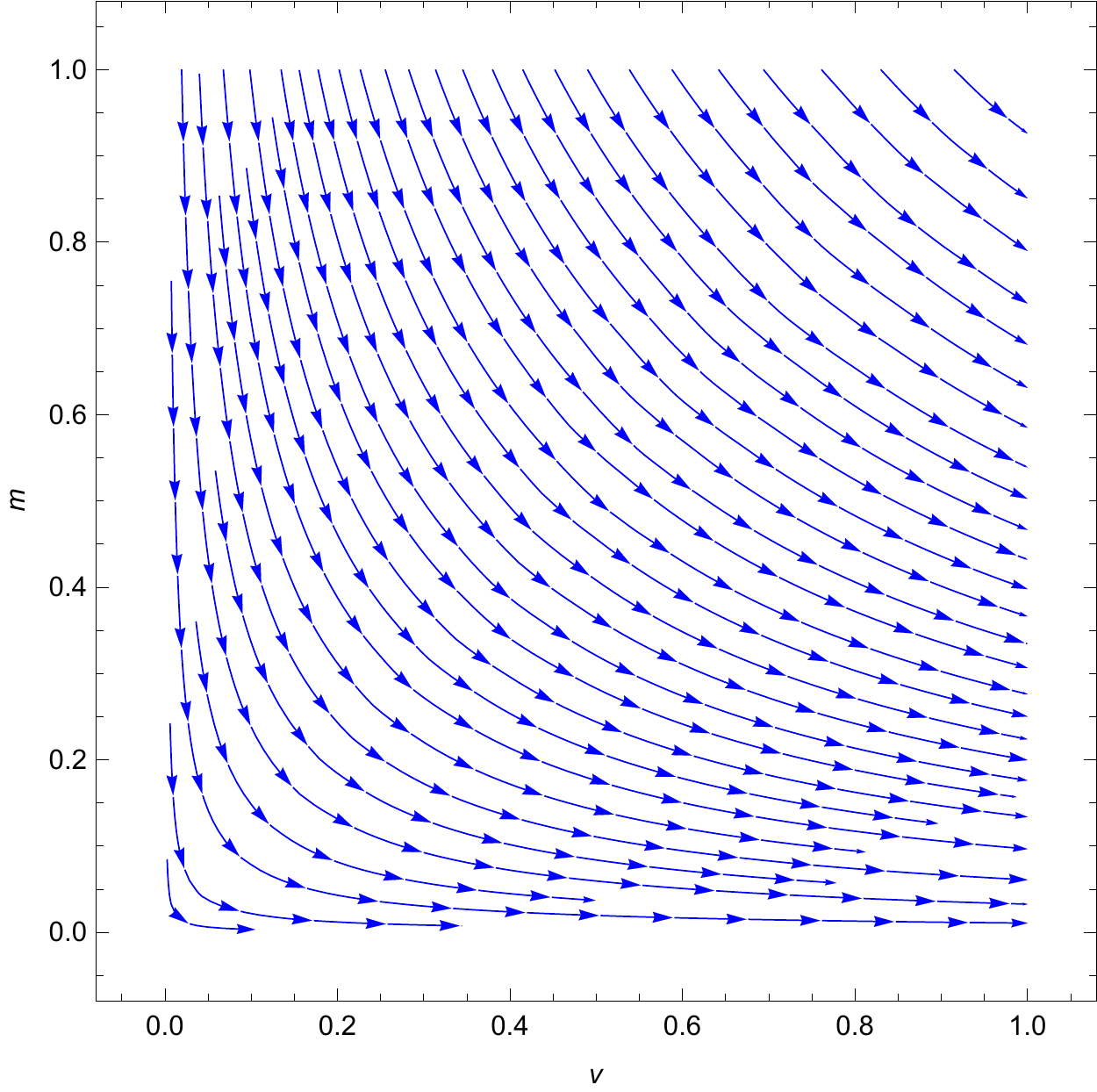}
        	} 
        	\subfigure[Level sets of $w$.]{\includegraphics[width=5cm,height=5cm]{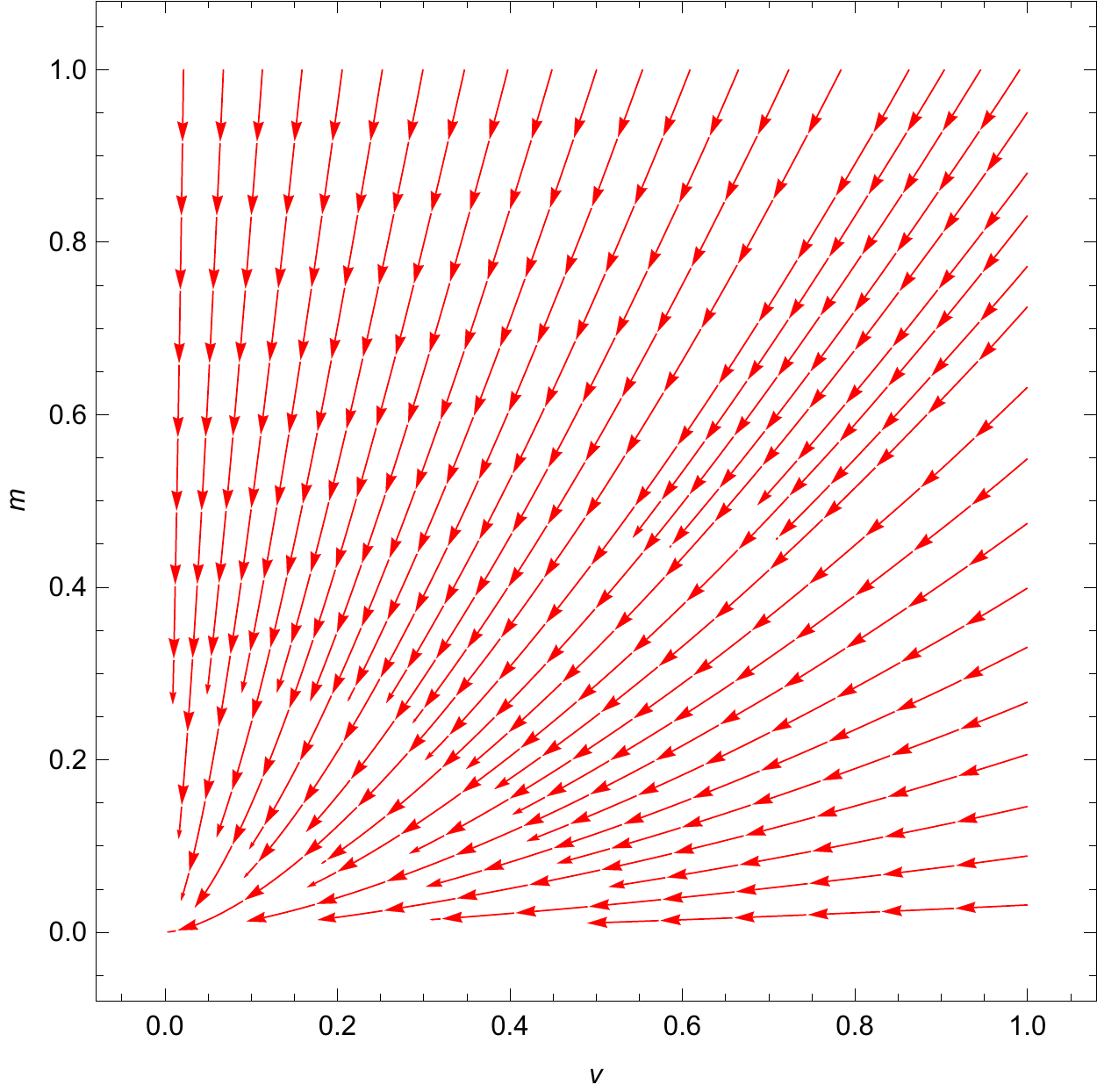}        }
        	\caption{Level sets of the Riemann invariats for \eqref{eq: cff mfg vis} with $\alpha=\frac 3 2$.}
        	 \label{lsts}
        	\end{figure}
In particular, $m$
satisfies
\[
m> M^{\frac{A(\alpha) r - B(\alpha) s}{(A(\alpha) - B(\alpha)) r s}}, 
\]
where the preceding lower bound is determined by the value of $m$ 
corresponding to $z=w=M$. 
	
Finally, if $v_0>0$, the condition $z<M$ gives that $v(x,t)>0$ for all $t$. 
 \end{proof}

Next, we combine the lower bound from the preceding lemma with the entropy
from Lemma \ref{elem} to improve the integrability of $v$. 

\begin{lem}
        Let $v,m \in C^2(\Tt \times (0,T_{\infty}))\cap C(\Tt \times (0, T_{\infty}))$ be the classical solution for \eqref{eq: cff mfg vis} and \eqref{icc} (we drop the index $\varepsilon$ for simplicity). Let $\eta$ be a smooth function satisfying  \eqref{eq: entropy/entflux 1}. Then,
         \begin{equation}\label{eq: time-derivative-entropy}
        \frac{d}{dt}\int_{\mathbb{T}} \eta(v, m) dx= -\varepsilon \int_{\mathbb{T}} (v_x, m_x)^T D^2\eta(v, m)(v_x, m_x)dx.
        \end{equation}
        Furthermore, assume that $m(\cdot,t)\in L^p(\Tt)$ for $0\leq t< T_{\infty}$ and that  $\eta(v,m)=v^am^b$ with $a>1$ and $b$ given by \eqref{b for convexity of eta}. Then, 
         \begin{equation}\label{eq: L^p-bound-on-v0}
         \int_{\mathbb{T}} v^{\frac{a p}{p-b}} (x,t) dx\leq     C
         \end{equation}
         for  $0\leq t< T_{\infty}$. 
\end{lem}        
        \proof
        Because
        $\eta$ is an entropy for \eqref{eq: cff mfg_ cl},
        we have
        \eqref{eq: time-derivative-entropy}.
        Moreover, according to Remark \ref{r22}, $\eta(v,m)=v^am^b$ is convex. Thus, \eqref{eq: time-derivative-entropy} implies that
        \begin{equation}\label{eq: bound-on-entropy}
        \int_{\mathbb{T}} v^a(x,t) m^b(x,t) dx\leq     \int_{\mathbb{T}} v^a(x,0) m^b(x,0) dx.
        \end{equation}
        Because $m(x,t)$ is bounded away from $0$  for all $0\leq t <T_{\infty}$  and $x\in \Tt$, by Proposition \ref{prop:invariant-dom} and because $b<0$, we have
                \begin{equation}\label{eq: L_infty_bound-on-m_inverse}
                  \Big\vert\Big\vert m^b(x,t) \Big\vert \Big\vert_{L^{\infty}(\mathbb{T})}\leq C_1.
                \end{equation}
        By H\"older's inequality, 
         \begin{equation}\label{eq: L^p-bound-on-v-by-holder}
        \int_{\mathbb{T}} v^{\frac{a p}{p-b}}(x,t) dx\leq  \left( \int_{\mathbb{T}} v^{a}m^{b}dx\right)^{\frac{p}{p-b}}  \left( \int_{\mathbb{T}} m^{p}(x,t) dx\right)^{\frac{-b}{p-b}}.
        \end{equation}
        Finally, we combine \eqref{eq: bound-on-entropy}, \eqref{eq: L_infty_bound-on-m_inverse}, and \eqref{eq: L^p-bound-on-v-by-holder} to obtain \eqref{eq: L^p-bound-on-v0}.
\endproof

By considering the limit $a\to \infty$, we obtain the following corollary.
\begin{cor}
\label{c33}
If
$m(\cdot,t)\in L^p(\Tt)$ for $0\leq t< T_{\infty}$ then 
\begin{equation}\label{eq: L^p-bound-on-v}
\int_{\mathbb{T}} v^{\tilde p} (x,t) dx\leq     C
\end{equation}
for all $\tilde p \leq \frac{p}{\theta(\alpha)}$, where $\theta(\alpha)$ is given by \eqref{tda}. 
\end{cor}
\begin{remark}
For $0<\alpha<2$, we have $0<\theta(\alpha)<1$. Hence, $\frac{p}{\theta(\alpha)}>p$. 
\end{remark}
\begin{proof}
The corollary follows by taking $a\to \infty$ and using \eqref{ limit of b/a}.
\end{proof}

        \begin{lem}
        	\label{l35}
                Let $v,m \in C^2(\Tt \times (0,T_{\infty}))\cap C(\Tt \times (0, T_{\infty}))$ be a classical solution of \eqref{eq: cff mfg vis}-\eqref{icc}. 
                Suppose
                the initial conditions in \eqref{icc} for \eqref{eq: cff mfg vis}
                take values in a compact subset set, $\mathcal{K}$, of $\mathcal{A}$.
                Then,
                \begin{equation}\label{estimate-deriv-int-m}
                \frac{d}{dt}\int_{\mathbb{T}} m^{r}(x,t) dx \leq C
                \end{equation}
         for all $1<r<\infty$ and all $0\leq t< T_\infty$. 
        \end{lem}
\proof
Using the second equation in \eqref{eq: cff mfg vis}, we have
$$
\begin{aligned}\label{s}
\frac{d}{dt}\int_{\mathbb{T}} m^{r}(x,t) dx & =\int_{\mathbb{T}}  r m^{r-1}m_t dx\\
                    &=\int_{\mathbb{T}} r m^{r-1}(v m^{1-\alpha} +\varepsilon m_x)_x dx.\\
\end{aligned}
$$
Integrating by parts, we get 
$$
\begin{aligned}\label{s}
\frac{d}{dt}\int_{\mathbb{T}} m^{r}(x,t) dx & =- r(r-1)\int_{\mathbb{T}} m^{r-2}m_x\left( vm^{1-\alpha} +\varepsilon m_x\right) dx\\
&=- r(r-1)\int_{\mathbb{T}} \left( m^{r-2}m_x v m^{1-\alpha}  +\varepsilon m^{r-2}m_x^2 \right) dx.
\end{aligned}
$$
It thus follows that 
\begin{equation}\label{eq: estimate-deriv-int-m}
\frac{d}{dt}\frac{1}{r(r-1)}\int_{\mathbb{T}} m^{r}(x,t) dx \leq\int_{\mathbb{T}} | m^{r-2}m_x v m^{1-\alpha}|  -\int_{\mathbb{T}}\varepsilon m^{r-2}m_x^2  dx.\\
\end{equation}
Next, 
by Cauchy inequality,
\begin{align}\label{eq: Cauchy-estimate}
\int_{\mathbb{T}}| m^{r-2}m_x v m^{1-\alpha}|  dx&= \int_{\mathbb{T}}| m^{r/2-1}m_x v m^{r/2-\alpha}|dx \\
 &\leq \frac{1}{2\varepsilon }\int_{\mathbb{T}} m^{r-2\alpha}v^2dx+\frac{1}{2}\int_{\mathbb{T}} \varepsilon m^{r-2}m_x^2dx. \nonumber
 \end{align}
 We combine \eqref{eq: estimate-deriv-int-m} and \eqref{eq: Cauchy-estimate} to obtain
\begin{equation}\label{eq: estimate-deriv- int-m-01}
\frac{d}{dt}\frac{1}{r(r-1)}\int_{\mathbb{T}} m^{r}(x,t) dx \leq C_\varepsilon \int_{\mathbb{T}} m^{r-2\alpha}v^2 dx-\frac{1}{2}\int_{\mathbb{T}}\varepsilon m^{r-2}m_x^2  dx.
\end{equation}
Now,
for $a>1$ and $b=b(\alpha, a)$ as in \eqref{b for convexity of eta},
we rewrite 
\begin{align*}
&\int_{\mathbb{T}} m^{r-2\alpha} v^2dx =
\int_{\mathbb{T}} m^{r-2\alpha-2\frac b a} (m^b v^a)^{\frac 2 a}\\
&\quad \leq
\left(
\int_{\mathbb{T}} m^{\left(r-2\alpha-2\frac b a\right)\frac{a}{a-2}}\right)^{\frac{a-2}a}
\left(
\int_{\mathbb{T}}
m^b v^a
\right)^{\frac{2}a}	\\
&\quad 
\leq C_{a}
\left(
\int_{\mathbb{T}} m^{\left(r-2\alpha-2\frac b a\right)\frac{a}{a-2}}\right)^{\frac{a-2}a}.
\end{align*}
Next, we notice that as $a\to \infty$, we have
\[
\left(r-2\alpha-2\frac b a\right)\frac{a}{a-2}\to
r-2\alpha+2 \theta(\alpha)<r. 
\]
Accordingly, for large enough $a$, 
\begin{equation}
\label{dgest}
\int_{\mathbb{T}} m^{r-2\alpha} v^2dx
\leq C \left(\int_{\mathbb{T}} m^r\right)^\mu,  
\end{equation}
for some $\mu<1$.

Next, we use Morrey's theorem to obtain that 
\begin{equation*}
\|m\|_{L^{\infty}}^r\leq C\left( 1 + \int_{\mathbb{T}} [(m^{r/2})_x]^2   dx\right).
\end{equation*}
Because
$\mu<1$, Young's inequality yields
\begin{equation*}
C^\mu\left( 1 + \int_{\mathbb{T}} [(m^{r/2})_x]^2   dx\right)^\mu\leq \frac{\varepsilon^2}{2C_0^2} \int_{\mathbb{T}} [(m^{r/2})_x]^2   dx +C_{\varepsilon,\mu}
\end{equation*}
for some $C_{\varepsilon,\mu}$. We combine \eqref{dgest} 
with the preceding inequalities to get
\begin{equation}\label{estimate-int-m-02}
 \int_{\mathbb{T}} m^{(r-2\alpha)} v^2dx \leq \frac{1}{4}\int_{\mathbb{T}}\varepsilon m^{r-2}m_x^2  dx+ C_{\varepsilon,\mu}.
\end{equation}  
Finally, we use \eqref{eq: estimate-deriv- int-m-01} and \eqref{estimate-int-m-02} to obtain
\begin{align}\label{estimate-deriv-int-m-02}
\frac{d}{dt}\frac{1}{r(r-1)}\int_{\mathbb{T}} m^{r}(x,t) dx &\leq  C_{\varepsilon,\mu}
-\frac{1}{4}\int_{\mathbb{T}}\varepsilon m^{r-2}m_x^2  dx.
\end{align}
The estimate \eqref{estimate-deriv-int-m} follows from \eqref{estimate-deriv-int-m-02}.
\endproof

Finally, we prove our main result, the existence of a global solution 
for \eqref{eq: cff mfg vis}.

\begin{teo}
\label{mteo}
Suppose
the initial conditions in \eqref{icc} for \eqref{eq: cff mfg vis}
take values in a compact subset set, $\mathcal{K}$, of $\mathcal{A}$
and
satisfy $w(v_0, m_0) < M$ and $z(v_0, m_0) < M$ for some $M>0$. 
Then, the maximal existence time $T_\infty$ is 
$T_\infty=+\infty$. 
\end{teo}
\begin{proof}
Suppose that the maximal existence time, $T_\infty$, satisfies $T_\infty<\infty$. 
First, we notice that by the conservation of mass, we have
\[
\sup_{0\leq t<T_\infty}\|m\|_{L^1}=1. 
\]
Thus, using first Lemma \ref{l35}, we obtain that 
$$
\sup_{0\leq t<T_\infty}\|m(\cdot, t)\|_{L^{p}(\Tt)}< C
$$
for all $p<\infty$.
Next, using Corollary \ref{c33}, we obtain
\[
\sup_{0\leq t<T_\infty}\|v(\cdot, t)\|_{L^{q}(\Tt)} <C
\]
for all $q<\infty$. Thus, the solution $(v,m)$ is bounded in $L^p\times L^q$
uniformly in $t$ up to $T_\infty$. Finally, a simple regularity argument for 
parabolic equations gives that the solution is classical up $T_\infty$ and,
thus, can be continued for $t>T_\infty$, which contradicts the maximality of $T_\infty$.
\end{proof}

 \section{Traveling waves}
 
 In this final section, we compute traveling waves for forward-forward MFGs and for MFGs with congestion with an anti-monotonic coupling. For forward-forward MFGs, the existence of traveling waves shows that these PDEs may fail to converge to a stationary solution.
 Similarly,
 for MFGs with congestion, the existence of traveling waves indicates that without monotonicity, convergence to a stationary solution may as well not hold. As far as we know, these are the first examples of traveling waves in MFGs.

\subsection{Traveling waves for forward-forward congestion MFGs} 
 
 We consider the following forward-forward congestion problem:
 \begin{equation}\label{power coupling FF}
  \begin{cases}
        v_t+\left( v^2/(2m^{\alpha}) -K m^{\alpha} \right) _x=0\\
        m_t-(m^{1-\alpha} v)_x=0,      
         \end{cases}
 \end{equation}
 with $K>0$. 
It is straightforward to check that for smooth initial data $m_0, v_0$ such that $v_0=c m_0^{\alpha}$
 with $c=\pm \sqrt{\frac{2 K}{3}}$,
  we have that
\[
m(x,t)=m_0(x+c t)\quad \hbox{and}\quad v(x,t)=c m^{\alpha}(x,t)
\]
solve \eqref{power coupling FF}.

\subsection{Traveling waves for non-monotonic MFGs with congestion} 

Now,  we consider the following non-monotonic MFG with congestion:
 \begin{equation}\label{MFG cong alpha}
 \begin{cases}
 -v_t+\left( v^2/(2m^{\alpha}) +K m^{\alpha} \right) _x=0\\
 m_t-(m^{1-\alpha} v)_x=0,        
 \end{cases}
 \end{equation}
where $K>0$. When $\alpha=1$, then 
 \begin{equation}\label{MFG cong alpha=1}
 \begin{cases}
 v_t+\left( v^2/(2m)+K m \right) _x=0\\
 m_t- v_x=0.        
 \end{cases}
 \end{equation}
 This congestion problem admits traveling wave solutions:
 \[ 
m(t,x)= m_0(x-\sqrt{2K}t)\quad\hbox{and}\quad v(t,x)= -v_0(x-\sqrt{2K}t)
\]
and
 \[ 
 m(t,x)= m_0(x+\sqrt{2K}t)\quad\hbox{and}\quad v(t,x)= v_0(x+\sqrt{2K}t),
 \]
which solve \eqref{MFG cong alpha=1} if the initial data   $(m_0, v_0)$  is smooth and satisfies $v_0=\sqrt{2K}m_0$.

\subsection{Potentials}

For $\alpha=1$, the two preceding MFGs can be converted into
to a scalar equation by introducing a potential function. In this last section, we record these remarkable equations. 

The second equation in \eqref{power coupling FF} 
becomes
\[
m_t-v_x=0. 
\]
Thus, we introduce a potential, $q$, such that $m=q_x$ and $v=q_t$. 
Accordingly, the first equation in  \eqref{power coupling FF} 
becomes
\[
-\left(K +\frac{ q_t^2}{2
        q_x^2}
\right)q_{xx}
+\frac{ q_t}{q_x}q_{xt}+q_{tt}=0,  
\]
which is a wave-type equation for $q$.

For the non-monotonic MFG, \eqref{MFG cong alpha=1} and $q$ such that $m=q_x$ and $v=q_t$, we have
\[
\left(K -\frac{ q_{t}^2}{2
        q_{x}^2}\right)q_{xx} +\frac{ q_{t}}{q_x}q_{xt}+q_{tt}=0.
\]



\end{document}